\documentclass{amcs}

\title{An adaptive observer design approach for discrete-time nonlinear systems}

\author[ad1][ad2]{Krishnan SRINIVASARENGAN}
\author[ad1][ad2]{Jos\'{e} RAGOT}
\author[ad1][ad2]{Christophe AUBRUN}
\author[ad1][ad2]{Didier MAQUIN}

\correspondingauthor{Krishnan SRINIVASARENGAN}

\address[ad1]{CRAN UMR 7039, Universit\'{e} de Lorraine, Vandoeuvre-l\`{e}s-Nancy, Cedex, France}
\address[ad2]{CRAN UMR 7039, CNRS, France \\ e-mail: \url{(krishnan.srinivasarengan, jose.ragot, christophe.aubrun,didier.maquin)@univ-lorraine.fr}}

\newcommand\bigzero{\makebox(0,0){\text{\Large0}}}

\newcommand\bigtranspose{\makebox(0,0){\text{\Large *}}}

\begin{document}
\begin{abstract}
We discuss a design approach for nonlinear discrete-time adaptive observer. This involves transforming a nonlinear system into a quasi-LPV (Linear Parameter Varying) polytopic model in Takagi-Sugeno (T-S) form using nonlinear embedding and sector nonlinearity (SNL) transformation. We then develop a discrete-time counterpart for a joint state and parameter estimation, based on design strategies developed for continuous time models in the existing literature. The design uses a Lyapunov approach and provides an error bounded by $\mathbb{L}_2$ gain. Based on this strategy, we propose a design for adaptive observers for nonlinear systems whose T-S form can have unmeasured premise variables.
\end{abstract}
\maketitle

\section{Introduction} \label{sec_introduction}
The term adaptive observers is used to represent joint state and parameter observers that estimate the states and parameters independently without augmenting them with the states \cite{ticlea_adaptive_2016}. For nonlinear systems, a systematic procedure for adaptive observers was first proposed in \cite{cho_systematic_1997}. These results were characterized further in \cite{besancon_remarks_2000}, where an adaptive observer form was introduced. However, the extension to discrete-time systems is not straightforward as it exploits the particular structure arising in continuous-time trajectories.

Other systematic adaptive observer design procedures attempt at specific applications such as fault diagnosis. In \cite{caccavale2008adaptive}, a diagnostic observer with an adaptive uncertainties estimation component is provided. All the states are assumed measured allowing the authors to use an innovation term as $\mathbf{e}_{x,k+1} - (A-K_o)\mathbf{e}_{x,k}$, where $A$ and $K_o$ are the system and the state observer gain matrices respectively with $\mathbf{e}_{x,k}$ representing the state error. This innovation allows for a cancellation of terms that complicate the adaptation of the approach in \cite{cho_systematic_1997} to discrete-time. %The main hurdle to adapt this approach for a general adaptive observer (apart from all states being measured) is the conditions on the eigenvalues of various matrices involved in the design process.
Another fault detection application based on the design of an adaptive observer is proposed in \cite{thumati_model_2008}. Here, the model uncertainties are bounded by known constants and hence it was possible to propose a state observer with a Luenberger form. 
%Both actuator and sensor fault modules, which estimate the parameter that the faults are characterized with, are part of the design process. 
The parameter estimation component of the observer is tailored for fault detection using terms for fault detection such as threshold and dead-zone. %This specific focus on fault detections makes the approach effective for fault diagnosis but difficult to adapt for general adaptive observers.
The focus on fault detection allows these works to tune the observers with specific constraints, hence generalizing them to a generic adaptive observer is ruled out.

%The attempts at discrete-time nonlinear observers in the literature points to a possible fundamental limitations due to the tools used for the analysis (e.g. quadratic Lyapunov function). 
One way to develop observers for general nonlinear systems could be using equivalent forms. A recurring theme in the above works in boundedness of inputs, matrix entries, etc. If we add them with boundedness of the states of the system, we can obtain consider using the linear parameter varying (LPV) or quasi-LPV system formulations. A nonlinear state equation of the form,
\begin{align}
\mathbf{x}_{k+1} = f(\mathbf{x}_k, \mathbf{u}_k)
\end{align}
could be put through a systematic procedure of factorizing them like that proposed in \cite{kwiatkowski_automated_2006} to obtain a quasi-LPV form as,
\begin{align}
\mathbf{x}_{k+1} = A(\mathbf{x}_k, \mathbf{u}_k) \mathbf{x}_k + B(\mathbf{x}_k, \mathbf{u}_k) \mathbf{u}_k
\end{align}
For a system with unknown parameters $\theta_j$ in the adaptive observer form of \cite{cho_systematic_1997}, the quasi-LPV form would be,
\begin{align}
\mathbf{x}_{k+1} = \sum_{j=1}^{n_\theta} \left(A_j(\mathbf{x}_k, \mathbf{u}_k) \mathbf{x}_k \theta_j + B_j(\mathbf{x}_k, \mathbf{u}_k) \mathbf{u}_k \theta_j\right) \label{eq_quasi_lpv_nonlin}
\end{align}
and possibly with an affine term. Here, $n_\theta$ is the number of parameters. The observer design strategy for such quasi-LPV systems could follow that of the linear time-varying (LTV) systems, if the matrices depend only on measured variables like outputs and inputs, as pointed out in \cite{ticlea_exponential_2013}.

In that direction, we can consider the adaptive observer proposed for LTV systems in \cite{guyader2003adaptive}, which takes inspiration from its continuous-time counterpart in \cite{zhang2002adaptive}. The authors propose an innovation term whose gain is obtained by filtering the parameters' transmission matrix in the state equation. This structure along with some boundedness assumption allow showing the exponential convergence of the observer. If the system has bounded zero mean noise, the estimation errors have an expected value that exponentially converges to zero. The main issue with this approach is the lack of clear procedure to choose a scalar that helps to guarantee convergence. These criticisms lead the authors in \cite{ticlea_adaptive_2016} to propose an exponential forgetting factor based approach adopted from \cite{ticlea_exponential_2013}. This approach mimics a Kalman filter with an update and propagation step, but has two interconnected exponential forgetting factor designs, thus preserving the adaptive observer structure. The main assumptions are complete uniform observability of the system and the invertibility of the system matrix of the LTV, $A_k,\ \forall k$.  % A scalar factor each to state and parameter estimation propagation steps ($\lambda_x, \lambda_\theta \in [0,1)$ respectively) allow for tuning the decay rate of the convergences independently. 

The LTV based observers, however, cannot handle the case when the system matrices depend on one of the unmeasured states. This type of design can be handled in the realm of one of the quasi-LPV polytopic models: Takagi-Sugeno (T-S) form. One way to obtain a T-S model, that exactly represents the original nonlinear system with a sector, is using the sector nonlinearity (SNL) transformation \cite{ohtake2003fuzzy}. For the quasi-LPV model in \eqref{eq_quasi_lpv_nonlin}, applying SNL would lead to a T-S model of the form,
\begin{align}
\mathbf{x}_{k+1} &= \sum_{i=1}^{2^{n_p}} \mu_i(\mathbf{z}_k) \left[ (A_i+\sum_{j=1}^{n_\theta} \bar{A}_{ij}\theta_j) \mathbf{x}_k  \right. \nonumber \\
  & \left.  \qquad \qquad \qquad \qquad + (B_i+\sum_{j=1}^{n_\theta} \bar{B}_{ij}\theta_j)) \mathbf{u}_k \right] \nonumber \\
\mathbf{y}_k &= C \mathbf{x}_k \label{eq_ts_with_theta_initial}
\end{align}
where, $n_p$ is the number of premise variables $\mathbf{z}_k$, which could be one of the states, inputs, and outputs. We consider a linear output equation, as in \cite{cho_systematic_1997}. We assert that the observer design for such models can cover all possible systems that are represented by those in \cite{cho_systematic_1997}. The observer for this type of system should take into account the fact that the weighting functions $\mu_i$ would be depending on estimated premise variables, rather than exact ones. Fortunately, there is a growing body of literature for observer design for T-S system with unmeasured premise variables \cite{lendek2010stability}, as well as using immersion techniques to avoid T-S systems with unmeasured premise variables \cite{ichalal2016method}.

%The design approach of \cite{cho_systematic_1997} were adapted for T-S system in Chapter 10 of \cite{lendek2010stability}. The focus of this approach was to estimate modeling uncertainties by characterising the uncertainties using a multiplication of known full column rank matrices and the uncertainty matrices. This approach was applied to a robotic arm in \cite{nagesh2012adaptive}, though all the states were assumed to be measured in this case. This adaptation still had the drawback of the need to find a Lyapunov matrix $P$ that meets some numerical/structural constraints which had to be independently checked outside the design procedure. This was improved in our work in \cite{srinivasarengan2017adaptive}, which provides an integration of the overall design procedure. However, a discrete-time adaptation of this adaptive observer design faced the same complication due to the stability analysis using quadratic Lyapunov function and hence could not be pursued. 

In this paper, we first develop a method for joint state and time-varying parameter estimation. The approach springs from the idea in \cite{bezzaoucha2013state2} to represent a time-varying parameter using SNL transformation. This was extended to T-S models with time-varying parameters in \cite{bezzaoucha2013state}. We derived the discrete-time version for the T-S models in \cite{srinivasarengan:hal-01399791}. In the present paper, we derive the time-varying parameter estimation for a linear system to illustrate the equivalence with that for the T-S models in  \cite{srinivasarengan:hal-01399791}. With this design approach, we propose an adaptive observer design for nonlinear systems whose T-S form has unmeasured premise variables. This would fill the gap that is left in the nonlinear discrete-time adaptive observers that cannot be solved by adapting LTV based design approaches.

This work builds on top of our communication in \cite{srinivasarengan:hal-01399791}. The key improvements are: to show the generalized nature of the results for both LTV and T-S system, providing extensions and refinement through corollaries, and to propose an adaptive observer design using the joint state and time-varying parameter estimation approach. Further, the illustration involves a different, but a more relevant example. The outline of the paper is as follows: the following section discusses the preliminaries that are used later in the paper. The Sec. \ref{sec_problem_formulation} discusses the model structure idea and formulates the problem. The joint state and time-varying parameter are derived for a linear time-varying system in Sec.\ref{sec_main_results}. These results are customized to design an adaptive observer design for nonlinear discrete-time systems in Sec. \ref{sec_adapt_observer_design}. A simulation example is given in Sec. \ref{sec_sim_example} to illustrate the proposed method. The paper is then summarized with a future outlook in Sec. \ref{sec_concluding_remarks}.

\section{Preliminaries} \label{sec_prelims}

\subsection{Notations}
Takagi-Sugeno models \cite{tanaka2004fuzzy} are of the form,
\begin{align}
\mathbf{x}_{k+1} &= \sum_{i=1}^r \mu_i(\mathbf{z}_k)\left[A_i \mathbf{x}_k + B_i \mathbf{u}_k \right] \nonumber\\
\mathbf{y}_k &= C\mathbf{x}_k \label{eq_ts_fund_model}
\end{align}
 Here, $r = 2^{n_p}$, where $n_p$ is the number of premise variables $\mathbf{z}_k$. The weighting functions $\mu_i(\mathbf{z}_k)$ capture the nonlinearity associated with the corresponding premise variables. Further, 
%Some key notations used in this work are
\begin{align}
\mathbf{x}_k \in \mathbb{R}^{n_x}, \ \mathbf{u}_k \in \mathbb{R}^{n_u}, \ \mathbf{z}_k \in \mathbb{R}^{n_p},\ \mathbf{y}_k \in \mathbb{R}^{n_y}\nonumber 
\end{align}
and 
\begin{align}
A_i \in \mathbb{R}^{n_x\times n_x}, \ \ B_i \in \mathbb{R}^{n_x\times n_u}, \ \ C \in \mathbb{R}^{n_y\times n_x}, \forall i
\end{align}
Given a symmetric matrix,
\begin{align*}
A = \begin{bmatrix}
a_{11} & a_{12} \\ * & a_{22}
\end{bmatrix}
\end{align*}
the '$*$' symbol represents the symmetric transpose element, that is, in this case, $*=a_{12}^T$.

\subsection{Preliminary Results}
The following known results would be referred to while proving the results in this work,
%\begin{lemma}{} \label{lemma_equivalence_lmi}
%The following LMIs, for appropriately suitable dimensions, are equivalent:
%\begin{align}
%A^TP A - Q < 0,\ P>0
%\end{align}
%\begin{align}
%\begin{bmatrix}
%-Q & A^T P \\
%PA & -P
%\end{bmatrix}
%< 0,\ P>0
%\end{align}
%This Lemma is referred from Chapter IV of \cite{kruszewski2006lois}.
%\begin{align}
%\exists\ G,\ \begin{bmatrix}
%-Q & A^T G \\
%G^T A & -G-G^T+P
%\end{bmatrix}
%< 0, \ P>0
%\end{align}
%\end{lemma}
\begin{lemma}{}{\cite{boyd1994linear}} \label{lemma_schur} For a symmetric matrix $M$, given by,
\begin{align*}
M = \begin{bmatrix}
A & B \\ B^T & C
\end{bmatrix}
\end{align*}
if $C$ is invertible, then the following properties hold:
\begin{enumerate}
\item $M > 0$ iff $ C > 0 \ and \ A-BC^{-1}B^T > 0$
\item if $C > 0$, then $M \geq 0 $ iff $A-BC^{-1} B^T \geq 0$.
\end{enumerate}
\end{lemma}
\begin{lemma}{}\cite{zhou1988robust}\label{lemma_timevaring_bounds}
Consider two matrices $X$ and $Y$ with appropriate dimensions, a time-varying matrix $\Delta_k$ and a positive scalar $\lambda$. The following property is verified:
\begin{align}
X^T \Delta^T_k Y + Y^T \Delta_k X \leq \lambda X^TX + \lambda^{-1} Y^TY
\end{align}
for $\Delta^T_k \Delta_k \leq I$
\end{lemma}
\begin{lemma}{}{\cite{de1992discrete}} \label{lemma_brl_discrete}
For a discrete-time system of the form,
\begin{align}
\mathbf{x}_{k+1} &= A \mathbf{x}_k + B \mathbf{u}_k \nonumber \\
\mathbf{y}_k &= C \mathbf{x}_k + D \mathbf{u}_k
\end{align}
the Bounded Real Lemma equivalent LMI condition for stability is,
{\small{
\begin{align}
&\qquad \qquad P=P^T>0, \nonumber \\ 
&\begin{bmatrix}
A^T P A-P +C^T C & A^T P B+C^TD \\
B^T P A + D^T C  & D^T D + B^TPB- \Gamma_2
\end{bmatrix}
\leq 0 \nonumber
\end{align}
}}
where $\Gamma_2$ is the $\mathbb{L}_2$ gain between the input $\mathbf{u}_k$ and the output $\mathbf{y}_k$.
\end{lemma}

\section{System Model and Problem Formulation} \label{sec_problem_formulation}
\subsection{Representing a time-varying parameter using SNL}
The idea of the estimation of time-varying parameter lies in representing it using the sector nonlinearity (SNL) transformation. The SNL assumes that the parameter is bounded and its boundary values are known. For a scalar parameter $\theta_k \in [\theta^1, \theta^2]$, we could write,
\begin{align}
\theta_k = \mu^1(\theta_k) \theta^{1} + \mu^2(\theta_k) \theta^{2} \label{eq_theta_SNL}
\end{align}
where,
\begin{align}
\mu^1(\theta_k) = \frac{\theta^2 - \theta_k}{\theta^2 - \theta^1}, \qquad \mu^2(\theta_k) = \frac{\theta_k - \theta^1}{\theta^2 - \theta^1}
\end{align}
The membership functions satisfy the convex sum property, that is,
\begin{align}
\sum_i \mu^{i}(.) = 1, \qquad 0 \leq \mu^{i}(.) \leq 1, \ \forall i \label{eq_convex_sum}
\end{align}
Hence each parameter could be represented by a weighted sum of two elements. %This could be extended to the vector case with $\theta \in \mathbb{R}^{n_\theta}$ and is illustrated in the following sections.
\begin{remark}{}\label{remark_vector_variables}
For a vector case, or for T-S systems with unknown parameters, the membership functions can be manipulated to obtain weighting functions that depend on the same variables. To illustrate this, take the case of two unknown parameters, $\theta_{1,k} \in [\theta_{1}^1, \theta_{1}^2]$ and $\theta_{2,k} \in [\theta_{2}^1, \theta_{2}^2]$ and represented as, 
\begin{align}
\theta_{1,k} &= \mu_{1}^1(\theta_{1,k}) \theta_{1}^1 + \mu_{1}^2(\theta_{1,k}) \theta_{1}^2 \nonumber \\
\theta_{2,k} &= \mu_{2}^1(\theta_{2,k}) \theta_{2}^1 + \mu_{2}^2(\theta_{2,k}) \theta_{2}^2 \label{eq_weighting_functions_vector}
\end{align}
We can now create a new formulation for the unknown parameters by,
\begin{align}
\theta_{1,k} &= \left(\mu_{2}^1(\theta_{2,k})+\mu_{2}^2 (\theta_{2,k})\right) \theta_{1,k} \nonumber \\
\theta_{2,k} &= (\mu_{1}^1(\theta_{1,k})+\mu_{1}^2 (\theta_{1,k})) \theta_{2,k}
\end{align}
By bringing the alternative form in \eqref{eq_weighting_functions_vector}, we obtain the variables that depend on the same, but 4 weighting functions, which are the products of the membership functions of the original representation. In general, this approach would lead to $2^{n_\theta}$ submodels, where $n_\theta$ is the number of parameters. A detailed treatment of this representation could be obtained from \cite{nagy2010systematic}.
\end{remark}
\subsection{System Model Structures}
Consider a linear time-varying system, the time varying nature is due to the unknown parameters $\Theta_k$ as follows:
\begin{align}
\mathbf{x}_{k+1} &= A(\Theta_k) \mathbf{x}_k + B(\Theta_k) \mathbf{u}_k \nonumber \\
\mathbf{y}_k &= C \mathbf{x}_k
\end{align}
where $\Theta_k \in \mathbb{R}^{n_\theta}$ is used to represent the vector of $\theta_i, \forall i\in 1,..,n_\theta$. We consider only a specific form of time-varying matrices that takes the following form:
%{\small{
\begin{align}
A(\Theta_k) &= A_0 + \sum_{i=1}^{n_\theta} \theta_{i,k} \bar{A}_i, \nonumber \\ 
B(\Theta_k) &= B_0 + \sum_{i=1}^{n_\theta} \theta_{i,k} \bar{B}_i \label{eq_sys_matrices_theta_form}
\end{align}
%}}
that is, it is possible to write the time-varying matrix as a sum of constant matrices that are scaled by the unknown parameter. We can use the SNL transformation as in \eqref{eq_theta_SNL}, to represent the matrices of time-varying parameters. Let us consider a scalar case,
\begin{align}
A(\theta_k) &= A_0 + \theta_k \bar{A} \nonumber \\
&= A_0 + (\mu^1(\theta_k) \theta^1 + \mu^2(\theta_k) \theta^2) \bar{A} \nonumber \\
&= \sum_{j=1}^2 \mu^j(\theta_k) (A_0 + \theta^j) \bar{A} 
\end{align}
with $\theta^j$ corresponding to one of $\theta^1$ or $\theta^2$ depending upon the submodel $j$. Similarly,
\begin{align}
B(\theta_k) &= \sum_{j=1}^2 \mu^j(\theta_k) (B_0 + \theta^j) \bar{B} 
\end{align}
This could then be extended to the vector case to yield,
\begin{align}
A(\Theta_k) &= \sum_{i=1}^{r} h_i(\Theta_k) (A_0+\theta^i \bar{A}_i), \nonumber \\ 
B(\Theta_k) &= \sum_{i=1}^{r} h_i(\Theta_k) (B_0+\theta^i \bar{B}_i)
\end{align}
where $r=2^{n_\theta}$. Here $h_i(\Theta_k)$ is the normalized product of a membership function $\mu_i^j(\theta_{i,k})$ of each parameter corresponding the submodel $i$ (See Remark \ref{remark_vector_variables}) and $\theta^i$ represents the sector boundary values of each parameter corresponding to the submodel $i$. More details could be obtained from \cite{tanaka2004fuzzy}. This would lead to
\begin{align}
\mathbf{x}_{k+1} &= \sum_{i=1}^{r} h_i(\Theta_k) (A_i \mathbf{x}_k + B_i \mathbf{u}_k) \nonumber \\
\mathbf{y}_k &= C \mathbf{x}_k \label{eq_sys_model_lin_final}
\end{align}
with
\begin{align}
A_i &= A_0 + \sum_{j=1}^{n_\theta} \theta_j^i \bar{A}_{i} \qquad  B_i = B_0 + \sum_{j=1}^{n_\theta} \theta_j^i \bar{B}_{i} \label{eq_sys_matrices_final}
\end{align}
where $\theta_j^i$ is the corresponding maximum or minimum value of $\theta_j$ for the submodel $i$. For the model in \eqref{eq_sys_model_lin_final}, we propose an observer of the form,
\begin{align}
\mathbf{\hat{x}}_{k+1} &= \sum_{i=1}^r h_i(\hat{\Theta}_k) (A_i \mathbf{\hat{x}}_k + B_i \mathbf{u}_k + L_i (\mathbf{y}_k - \mathbf{\hat{y}}_k)) \nonumber \\
\hat{\Theta}_{k+1} &= \hat{\Theta}_k + \sum_{i=1}^r h_i(\hat{\Theta}_k) (K_{y,i} (\mathbf{y}_k - \mathbf{\hat{y}}_k) - K_{\theta} \hat{\Theta}_k) \nonumber \\
\mathbf{\hat{y}}_k &= C \hat{x}_k \label{eq_obs_model_lin_final}
\end{align}
The gains $L_{i} \in \mathbb{R}^{n_x\times n_y}$ and $K_{y,i} \in \mathbb{R}^{n_\theta \times n_y}$ are to be estimated while the gain $K_{\theta} \in \mathbb{R}^{n_\theta \times n_\theta}$ is chosen. The choice of $K_{\theta}$ shall typically be in the form of a diagonal matrix. In the initial work \cite{bezzaoucha2013state}, this was introduced to avoid a marginal stability condition for the error dynamics. As discussed in \cite{srinivasarengan2016nonlinear}, choosing this reduces the number of variables in the final LMI to solve and hence allows for a computationally tractable problem. Further, in the discrete-time case, $K_\theta$ as a variable leads to unresolvable nonlinear terms in the matrix inequalities.

\subsection{Uncertain-like model representation}
Let us define the state estimation error $\mathbf{e}_{\mathbf{x},k} = \mathbf{x}_k - \mathbf{\hat{x}}_k$. If we want to analyse the dynamics of the errors based on the system and observer models in \eqref{eq_sys_model_lin_final} and \eqref{eq_obs_model_lin_final}, it would involve comparing systems weighted by functions that depend on mismatched variables (i.e., $\mathbf{x}_k, \Theta_k$ vs $\mathbf{\hat{x}}_k, \hat{\Theta}_k$). This is a typical problem in observer design for T-S systems with unmeasured premise variables. There are different approaches to deal with it. In this work, we use the approach proposed in \cite{ichalal2009state} to develop an uncertain-like model representation. By making use of the convex sum property in \eqref{eq_convex_sum}, we can rewrite \eqref{eq_sys_model_lin_final}, without making any approximations, as,
\begin{align}
\mathbf{x}_{k+1} &= \sum_{i=1}^{r} h_i(\hat{\Theta}_k) (A_i \mathbf{x}_k + B_i \mathbf{u}_k) \nonumber \\ 
& \qquad + \sum_{i=1}^{r} (h_i(\Theta_k) - h_i(\hat{\Theta}_k)) (A_i \mathbf{x}_k + B_i \mathbf{u}_k) \nonumber \\
\mathbf{y}_k &= C \mathbf{x}_k
\end{align}
Let us denote,
\begin{align}
\Delta A_k &= \sum_{i=1}^r(h_i(\Theta_k) - h_i(\hat{\Theta}_k)) A_i = \mathcal{A} \Sigma_{A,k} E_A \nonumber \\
\Delta B_k &= \sum_{i=1}^r(h_i(\Theta_k) - h_i(\hat{\Theta}_k)) B_i = \mathcal{B} \Sigma_{B,k} E_B \label{eq_matrix_representation_uncertain_term}
\end{align} 
where, 
\begin{align}
\mathcal{A} &= \begin{bmatrix} A_1 & A_2 & ... & A_r \end{bmatrix} \in \mathbb{R}^{n_x\times n_xr} \nonumber \\
E_A &= \begin{bmatrix} I_{n_x} & I_{n_x} & ... & I_{n_x}\end{bmatrix}^T \in \mathbb{R}^{n_xr \times n_x} \nonumber\\
\Sigma_{A,k} &= \begin{bmatrix}
(h_1 - \hat{h}_1)I_{n_x} &  &  & \bigzero  \\
 &  & \cdots & \\
   \bigzero& & & (h_r - \hat{h}_r)I_{n_x}\\
\end{bmatrix} \label{eq_matrices_for_uncertain_form_A}
\end{align}
Similarly,
\begin{align}
\mathcal{B} &= \begin{bmatrix} B_1 & B_2 & ... & B_r \end{bmatrix} \in \mathbb{R}^{n_u\times n_ur} \nonumber \\
E_B &= \begin{bmatrix} I_{n_u} & I_{n_u} & ... & I_{n_u}\end{bmatrix}^T \in \mathbb{R}^{n_ur \times n_u} \nonumber\\
\Sigma_{B,k} &= \begin{bmatrix}
(h_1 - \hat{h}_1)I_{n_u} &  &  & \bigzero  \\
 &  & \cdots & \\
   \bigzero& & & (h_r - \hat{h}_r)I_{n_u}\\
\end{bmatrix} \label{eq_matrices_for_uncertain_form_B}
\end{align}
with $h_i$ and $\hat{h}_i$ stand for $h_i(\Theta)$ and $\hat{h}_i(\Theta)$ respectively. Noting that $-1 \leq (h_i - \hat{h}_i) \leq 1$, the matrices $\Sigma_{A,k} \in \mathbb{R}^{n_xr\times n_xr}$, $\Sigma_{B,k} \in \mathbb{R}^{n_ur\times n_ur}$ have the useful property,
\begin{align}
\Sigma_{A,k}^T \Sigma_{A,k} \leq I, \qquad \Sigma_{B,k}^T \Sigma_{B,k} \leq I \label{eq_scalar_product_timevarying_terms}
\end{align}
which will later be used to bound the time-varying difference between the known and estimated weighting functions. This will lead to the system model \eqref{eq_sys_model_lin_final} be represented as,
\begin{align}
\mathbf{x}_{k+1} &= \sum_{i=1}^r h_i(\hat{\Theta}) [(A_i+\Delta A_k)\mathbf{x}_k + (B_i+\Delta B_k)\mathbf{u}_k] \nonumber \\
\mathbf{y}_k &= C \mathbf{x}_k \label{eq_sys_model_lin_final_uncertain}
\end{align}
As the model \eqref{eq_sys_model_lin_final_uncertain} and its observer \eqref{eq_obs_model_lin_final} now share the same weighting functions $h_i(\hat{\Theta})$, it's therefore possible to express in a more simple and tractable form the state and the parameter estimation errors.
\section{Joint state and time-varying parameter estimation} \label{sec_main_results}
In this section, we provide the results for the stability analysis of the design of joint state and parameter observer. The results can be considered the discrete-time version of the observer design in \cite{bezzaoucha2013state} and follows the steps in \cite{srinivasarengan:hal-01399791}.
\begin{theorem}{} \label{thm_linear_main}
Given the system model of the form \eqref{eq_sys_model_lin_final}, there exists an observer of the form \eqref{eq_obs_model_lin_final}, if there exists $P_0$, $P_1$, $R_{i}$, $F_{i}$, $\lambda_1$, $\lambda_2$, $\lambda_3$, $\lambda_4$, $\Gamma_2^j$ ($\forall i \in [1,r], \forall j\in \{x,u,\theta,\Delta\theta\}$), such that,
\begin{align}
&P_0 = P_0^T > 0, \ P_1 = P_1^T > 0 \nonumber \\
&\lambda_m>0,\ \forall m \in \{1,2,3,4\}, \ \Gamma_2^j>0, \forall j
\end{align}
\begin{align}
\begin{bmatrix}
-P+I & Q_{A,i}    & \Phi_i^T P & 0      &  \\
  *  & T_{22} &  0         & Q_B^T  &  T_{\mathcal{A}\mathcal{B}} \\
  *  &  *     &   -P       & 0      &  \\
  *  &  *     &    *       & -P     &  \\
     & \bigtranspose &          &       & \Lambda \\
\end{bmatrix} < 0 \label{eq_main_lmi_condition}
\end{align}
Here,
%{\scriptsize{ 
\begin{align}
P_{0,i} &= A_{i}^T P_0-C^T R_{i}^T,  \ \ P = diag(P_0, P_1), \nonumber \\ 
\Phi_{i}^T P &= \begin{bmatrix} P_{0,i}  & -C^T F^T_{i} \\ 0  & - K_{\theta}^T P_1 \end{bmatrix} \nonumber \\
Q_{A,i} &= \begin{bmatrix}
0 & 0 & -C^TF_{i}^T (I+K_{\theta}) & -C^T F_{i}^T \\
& & & & \\
0 & 0 & -K_{\theta}^T P_1 (I+K_{\theta}) & -K_{\theta}^T P_1
\end{bmatrix} \nonumber
\\ 
Q_{B} &= \begin{bmatrix}
0 & 0 \\ 
0 & 0 \\
0 & (I+K_{\theta})^T P_1 \\ 
0 & P_1
\end{bmatrix} \nonumber \\
T_{22} &= 
\begin{bmatrix}
T_{22}^{11} & 0 & 0 & 0 \\
0 & T_{22}^{22} & 0 & 0 \\
0 & 0 & -\Gamma_2^\theta & 0 \\
0 & 0 & 0 & -\Gamma_2^{\Delta \theta}
\end{bmatrix}\nonumber \\
\Lambda &= diag(-\lambda_1I, -\lambda_3I, -\lambda_2I, -\lambda_4I) \nonumber \\
T_{\mathcal{A}\mathcal{B}} &= \begin{bmatrix} 
P_{0,i} \mathcal{A} & P_{0,i} \mathcal{B} & 0 & 0 \\
0   &   0 & 0 & 0 \\
0   &   0 & 0 & 0 \\
0   &   0 & P_0 \mathcal{A} & P_0 \mathcal{B}
\end{bmatrix} \label{eq_theorem_term_defs}
\end{align}
%}}
where, $T_{22}^{11} = -\Gamma_2^x+(\lambda_1+\lambda_2)E_A^TE_A$ and $T_{22}^{22} = -\Gamma_2^u+(\lambda_3+\lambda_4)E_B^TE_B$. The observer gains are given by,
\begin{align}
K_{y,i} = P_1^{-1} F_{i}, \qquad L_{i} = P_0^{-1} R_{i}
\end{align}
\end{theorem}
\begin{proof}{}
Consider the uncertain-like representation of the system in \eqref{eq_sys_model_lin_final_uncertain}. Comparing it with the observer \eqref{eq_obs_model_lin_final}, and defining errors of the form,
\begin{align*}
\mathbf{e}_{\mathbf{x},k} = \mathbf{x}_k - \mathbf{\hat{x}}_k, \ \ \ \mathbf{e}_{\Theta,k} = \Theta_k -\hat{\Theta}_k
\end{align*}
we get,
{\small{
\begin{align}
\mathbf{e}_{\mathbf{x},k+1} &= \sum_{i=1}^r h_i(\hat{\Theta}_k) [(A_{i}-L_{i} C) \mathbf{e}_{\mathbf{x},k} + \Delta A_k \mathbf{x}_k + \Delta B_k \mathbf{u}_k] \nonumber \\
\mathbf{e}_{\Theta,k+1} &= \sum_{i=1}^r h_i(\hat{\Theta}_k)[\Delta \Theta_k + (I+K_{\theta})\theta_k \nonumber \\
 & \qquad \qquad \qquad \qquad  - K_{y,i} C \mathbf{e}_{\mathbf{x},k} - K_{\theta} \mathbf{e}_{\Theta,k}]
\end{align}
}}%
where, $\Delta \Theta_k = \Theta_{k+1} - \Theta_k$. We can represent this error dynamics as,
\begin{align}
\begin{bmatrix} \mathbf{e}_{\mathbf{x},k+1} \\ \\ \mathbf{e}_{\Theta,k+1}\end{bmatrix} 
= 
\Phi_i \begin{bmatrix} \mathbf{e}_{\mathbf{x},k} \\ \\ \mathbf{e}_{\Theta,k}\end{bmatrix}
+
\Psi_{i,k} \begin{bmatrix} \mathbf{x}_k \\ \mathbf{u}_k \\ \Theta_k \\ \Delta \Theta_k \end{bmatrix}
\end{align}
where,
\begin{align}
\Phi_{i} &= \begin{bmatrix}
A_{i} - L_{i} C & 0 \\ -K_{y,i} C & -K_{\theta}
\end{bmatrix}
\nonumber \\
\Psi_{i,k} &= \begin{bmatrix}
\Delta A_k & \Delta B_k & 0 & 0 \\
0 & 0 & I+K_{\theta} & I
\end{bmatrix} \label{eq_phi_psi}
\end{align}
Considering, 
\begin{align*}
\mathbf{e}_{a,k} = \begin{bmatrix} \mathbf{e}_{\mathbf{x},k}^T & \mathbf{e}_{\Theta,k}^T \end{bmatrix}^T, \ 
\tilde{u}_k = \begin{bmatrix} \mathbf{x}_k^T & \mathbf{u}_k^T & \Theta_k^T & \Delta \Theta_k^T \end{bmatrix}^T
\end{align*}
the error dynamics can be written by,
\begin{align}
\mathbf{e}_{a,k+1} = \Phi_i \mathbf{e}_{a,k} + \Psi_{i,k} \tilde{u}_k \label{eq_error_dyn_augmented}
\end{align}
The aim here is the asymptotic decay of the error and the minimization of the effect of $\tilde{u}_k$ on the error. It is to be noted that $\Phi_i$ has constant entries, but $\Psi_{i,k}$ has time varying entries. To analyze the stability of \eqref{eq_error_dyn_augmented}, we consider the following Lyapunov candidate,
\begin{align}
V_k = \mathbf{e}_{a,k}^T P \mathbf{e}_{a,k}
\end{align}
Since there are time-varying perturbations that affect the error $\mathbf{e}_{a,k}$ in \eqref{eq_error_dyn_augmented}, the sufficient condition for stability that we consider is,
\begin{align}
V_{k+1} < V_k - (\mathbf{e}_{a,k}^T \mathbf{e}_{a,k} - \tilde{u}_k^T \Gamma_2 \tilde{u}_k) \label{eq_lyapunov_trajectory}
\end{align}
where $\Gamma_2$ is a block diagonal matrix with the entries 
\begin{align}
\Gamma_2 = diag(\Gamma_2^x, \Gamma_2^u, \Gamma_2^\theta, \Gamma_2^{\Delta \theta})
\end{align}
that represent the $\mathbb{L}_2$-gains of the effect of the elements in $\tilde{u}$ on the error, respectively. By applying the discrete-time version of the bounded real lemma (BRL) in Lemma \ref{lemma_brl_discrete}, we get the matrix inequality condition,
\begin{align}
\begin{bmatrix}
\Phi_{i}^T P \Phi_{j} - P + I  &  \Phi_{i}^T P \Psi_{i,k} \\
*  & \Psi_{i,k}^T P \Psi_{j,k}-\Gamma_2
\end{bmatrix}
< 0 \label{eq_quadratic_form_after_BRL}
\end{align}
The introduction of $j$ terms is to illustrate that we have cross terms between the different submodels. However, we could take the more conservative condition of considering $j=i$, based on the illustrations in Theorem 17 in \cite{blanco2001stabilisation}. We find another form for \eqref{eq_quadratic_form_after_BRL}, so as to,
\begin{itemize}
\item Obtain linear bounds for the nonlinearities (in $\Phi_{i}^T P \Phi_{j}$, $\Phi_i^T P \Psi_{i,k}$ and their transposes)
\item Obtain bounds for the time-varying terms (in $\Phi_i^T P \Psi_{i,k}$ and $\Psi_{i,k}^T P \Psi_{j,k}$)
\end{itemize}
\paragraph{Reducing nonlinearities} The quadratic terms associated with $\Phi_i$ and $\Psi_{i,k}$ could be reduced to linear terms. By using the Schur complements (Lemma \ref{lemma_schur}) for the nonlinear terms, the matrix terms in \eqref{eq_quadratic_form_after_BRL} could be reduced to, 
\begin{align} \label{eq_quadratic_form_after_BRL_simplified}
\begin{bmatrix}
-P+I & \Phi_{i}^T P \Psi_{i,k} & \Phi_{i}^T P & 0  \\
  *    & -\Gamma_2 & 0 & \Psi_{i,k}^T P \\
  *    & * & -P & 0 \\
  *    & * & * & -P
\end{bmatrix} < 0
\end{align}
This has not resolved all the nonlinear entries, though, and the residual factors are in the form of unresolvable terms inside $\Phi_i^T P$ and $\Phi_i^T P \Psi_{i,k}$. This is because as in \eqref{eq_phi_psi}, $\Phi_i$ has two variables $L_i$ and $K_{y,i}$, as part of the matrix split into $n_x$ and $n_\theta$ blocks. This issue is alleviated in two steps:
\begin{itemize}
\item Consider a diagonal structure for the Lyapunov matrix $P = \begin{bmatrix} P_0 & 0 \\ 0 & P_1 \end{bmatrix}$
\item This Lyapunov structure would lead to terms $P_0 L_i$ and $P_1 K_{y,i}$. These quadratic terms are eliminated by introducing new variables, 
\begin{align}
R_i = P_0 L_i, \qquad F_i = P_1 K_{y,i}
\end{align}
\end{itemize}
These steps would reduce the nonlinear matrix entries in \eqref{eq_quadratic_form_after_BRL_simplified} to linear terms. First, we define, to simplify notations,
\begin{align}
P_{0,i} = A_i^TP_0 - C^TR_i^T
\end{align}
The term $\Phi_{i}^T P$ would reduce to,
\begin{align}
\Phi_{i}^T P &= \begin{bmatrix} P_{0,i}  & -C^T F^T_{i} \\ 0  & - K_{\theta}^T P_1 \end{bmatrix} \nonumber
\end{align}
Further, we split the linearized time-varying matrices to those with constant entries and time-varying terms,
\begin{align}
\Phi_{i}^T P \Psi_{i,k} &= Q_{A,i} + \mathcal{L}_{U,i,k} \nonumber \\
\Psi_{i,k}^T P &= Q_B + \mathcal{L}_{L,k}^T
\end{align}
where, % $Q_{A,i}$ and $Q_B$ are given in \eqref{eq_theorem_term_defs} and
\begin{align}
Q_{A,i} &= \begin{bmatrix}
0 & 0 & -C^TF_{i}^T (I+K_{\theta}) & -C^T F_{i}^T \\
& & & & \\
0 & 0 & -K_{\theta}^T P_1 (I+K_{\theta}) & -K_{\theta}^T P_1
\end{bmatrix} \nonumber \\ 
Q_{B} &= \begin{bmatrix}
0 & 0 \\ 
0 & 0 \\
0 & (I+K_{\theta})^T P_1 \\ 
0 & P_1
\end{bmatrix} \nonumber \\
\mathcal{L}_{U,i,k} &= \begin{bmatrix}
P_{0,i} \Delta A_k & P_{0,i} \Delta B_k & 0 & 0 \\
& & & & \\
0 & 0 & 0 & 0
\end{bmatrix} \nonumber \\
\mathcal{L}_{L,k} &= \begin{bmatrix}
P_0 \Delta A_k & P_0 \Delta B_k & 0 & 0 \\
  0 & 0 & 0 & 0
\end{bmatrix} \label{eq_timevarying_matrices_L}
\end{align}

\paragraph{Bounds for time-varying terms}
The linearized version of the inequality in \eqref{eq_quadratic_form_after_BRL_simplified} can now be split into terms with and without time-varying terms and their corresponding transposes,
\begin{align}
Q_i + \mathcal{L}_{i,k} + \mathcal{L}_{i,k}^T < 0 \label{eq_lmi_split_constant_timevarying}
\end{align}
where,
\begin{align}
Q_{i} = 
\begin{bmatrix}
-P+I & Q_{A,i} & \Phi_{i}^T P & 0  \\
   * & -\Gamma_2 & 0 & Q_{B} \\
   * &   * & -P & 0 \\
   * &   * &  * & -P
\end{bmatrix} 
\end{align}
with $Q_{A,i}$ and $Q_B$ given in \eqref{eq_theorem_term_defs}. The time-varying terms are gathered as below, 
\begin{align}
\mathcal{L}_{i,k} &= \begin{bmatrix}
0 & \mathcal{L}_{U,i,k}& 0 & 0 \\
0 & 0   & 0 &  0 \\
0  & 0 & 0 & 0 \\
0 & \mathcal{L}_{L,k}& 0 & 0
\end{bmatrix} \label{eq_time_varying_terms_matrix}
\end{align}
and its transpose with the terms defined in \eqref{eq_timevarying_matrices_L}. There are 4 time-varying terms and their transposes in \eqref{eq_lmi_split_constant_timevarying}. In \eqref{eq_matrix_representation_uncertain_term}, we showed that the uncertain-like terms could be written as a product of matrices and further showed an interesting property of the time-varying matrix in \eqref{eq_scalar_product_timevarying_terms}. In the same lines, we can split each of the uncertain-like term in $\mathcal{L}_{i,k}$. Let us denote the four terms as part of individual matrices, $\mathcal{L}_{A1,k}$, $\mathcal{L}_{A2,k}$, $\mathcal{L}_{B1,k}$, $\mathcal{L}_{B2,k}$ corresponding to the time-varying factors, $P_{0,i} \Delta A_k$, $P_0 \Delta A_k$, $P_{0,i} \Delta B_k$, $P_0 \Delta B_k$ respectively. That is,
\begin{align}
\mathcal{L}_{A1,k} &= 
\begin{bmatrix} 
0 & \begin{bmatrix} P_{0,i} \Delta A_k & 0 & 0 \\ 0 & 0 & 0 \end{bmatrix} & 0 & 0 \\ 
0 & 0 & 0 & 0\\ 
0 & 0 & 0 & 0\\ 
0 & 0 & 0 & 0  
\end{bmatrix} \nonumber \\
\mathcal{L}_{B2,k} &= \begin{bmatrix} 
0 & 0 & 0 & 0  \\ 
0 & 0 & 0 & 0 \\  
0 & 0 & 0 & 0 \\ 
0 & \begin{bmatrix} 0 & P_0 \Delta B_k & 0 \\ 0 & 0 & 0 \end{bmatrix} & 0 & 0  \end{bmatrix}
\end{align}
and similarly for $\mathcal{L}_{B1,k}$ and $\mathcal{L}_{A2,k}$, so that,
\begin{align}
\mathcal{L}_{i,k} = \mathcal{L}_{A1,k} + \mathcal{L}_{A2,k}  + \mathcal{L}_{B1,k} + \mathcal{L}_{B2,k}
\end{align} 
It is to be noted that the $0$ entries in the matrices have appropriate dimension and are usually grouped together to make the representation easier. By taking cue from the representation in \eqref{eq_matrix_representation_uncertain_term}, we represent the uncertain-like terms as, \par
{\footnotesize
\begin{align}
\mathcal{L}_{A1,k} &= \begin{bmatrix} \begin{bmatrix} P_{0,i} \mathcal{A} \\ 0 \end{bmatrix}\\ 0  \\ 0 \\ 0 \end{bmatrix} \Sigma_{A,k} \begin{bmatrix}
0 & \begin{bmatrix} E_A & 0 & 0 \end{bmatrix} & 0 & 0
\end{bmatrix} \nonumber \\
\mathcal{L}_{A2,k} &= \begin{bmatrix} 0  \\ 0 \\ 0 \\ \begin{bmatrix} P_0 \mathcal{A} \\ 0 \end{bmatrix}\end{bmatrix} \Sigma_{A,k} \begin{bmatrix}
0 & \begin{bmatrix} E_A & 0 & 0 \end{bmatrix} & 0 & 0
\end{bmatrix} \nonumber \\
\mathcal{L}_{B1,k} &= \begin{bmatrix} \begin{bmatrix} P_{0,i} \mathcal{B} \\ 0 \end{bmatrix}\\ 0  \\ 0 \\ 0 \end{bmatrix} \Sigma_{B,k} \begin{bmatrix}
0 & \begin{bmatrix} 0 & E_B & 0 \end{bmatrix} & 0 & 0
\end{bmatrix} \nonumber \\
\mathcal{L}_{B2,k} &= \begin{bmatrix} 0  \\ 0 \\ 0 \\ \begin{bmatrix} P_0 \mathcal{B} \\ 0 \end{bmatrix}\end{bmatrix} \Sigma_{B,k} \begin{bmatrix}
0 & \begin{bmatrix} 0 & E_B & 0 \end{bmatrix} & 0 & 0
\end{bmatrix} %\nonumber
\end{align}
}
Now, we apply Lemma \ref{lemma_timevaring_bounds} on the sum of these terms and their transposes. \par
{\footnotesize{
\begin{align}
&\mathcal{L}_{A1,k}+ \mathcal{L}^T_{A1,k}
\leq 
\lambda_1^{-1} 
\begin{bmatrix} \begin{bmatrix} P_{0,i} \mathcal{A} \\ 0 \end{bmatrix}\\ 0  \\ 0 \\ 0 \end{bmatrix}
\begin{bmatrix} \begin{bmatrix} \mathcal{A}^T P_{0,i}  & 0 \end{bmatrix} & 0  & 0 & 0 \end{bmatrix}
\nonumber \\
&\qquad\qquad\qquad \qquad+\lambda_1 
\begin{bmatrix} 0 \\ \begin{bmatrix} E_A^T \\ 0 \\ 0 \end{bmatrix} \\ 0 \\ 0\end{bmatrix} 
\begin{bmatrix}
0 & \begin{bmatrix} E_A & 0 & 0 \end{bmatrix} & 0 & 0
\end{bmatrix} \nonumber \\
&\leq 
\begin{pmatrix} 
\begin{bmatrix} \lambda_1^{-1} P_{0,i} \mathcal{A} \mathcal{A}^T P_{0,i} & 0 \\ 0 & 0 \end{bmatrix} & 0 & 0 & 0 \\ 
0 & \begin{bmatrix} 0 & \lambda_1 E_A^T E_A  & 0 \\ 0 & 0 & 0 \end{bmatrix} & 0 & 0 \\ 
0 & 0 & 0 & 0 \\ 
0 & 0 & 0 & 0 
\end{pmatrix} 
\end{align}
}}%
for some scalar $\lambda_1$. Similarly the sums of $\mathcal{L}_{A2,k} + \mathcal{L}^T_{A2,k}$, $\mathcal{L}_{B1,k} + \mathcal{L}^T_{B1,k}$ and $\mathcal{L}_{B2,k} + \mathcal{L}^T_{B2,k}$ are bounded (respectively) by,
{\footnotesize{
\begin{align}
\begin{pmatrix} 
0 & 0 & 0 & 0 \\ 
0 & \begin{bmatrix} \lambda_2 E_A^T E_A  & 0 & 0 \\ 0 & 0 & 0 \end{bmatrix} & 0 & 0 \\ 
0 & 0 & 0 & 0 \\ 
0 & 0 & 0 & \begin{bmatrix} \lambda_2^{-1} P_0 \mathcal{A} \mathcal{A}^T P_0 & 0 \\ 0 & 0 \end{bmatrix} 
\end{pmatrix} \nonumber \\
\begin{pmatrix} 
\begin{bmatrix} \lambda_3^{-1} P_{0,i} \mathcal{B} \mathcal{B}^T P_{0,i} & 0 \\ 0 & 0 \end{bmatrix} & 0 & 0 & 0 \\ 
0 & \begin{bmatrix} 0 & 0 & 0 \\ 0 & \lambda_3 E_B^T E_B & 0\end{bmatrix} & 0 & 0  \\ 
 & & & \\
0 & 0 & 0 & 0 \\ 
0 & 0 & 0 & 0 
\end{pmatrix} \nonumber
\end{align}
\begin{align}
\begin{pmatrix} 
0 & 0 & 0 & 0 \\ 
0 & \begin{bmatrix} 0 & 0 & 0 \\ 0 & \lambda_4 E_B^T E_B & 0\end{bmatrix}& 0 & 0 \\ 
0 & 0 & 0 & 0 \\ 
0 & 0 & 0 & \begin{bmatrix} \lambda_4^{-1} P_0 \mathcal{B} \mathcal{B}^T P_0 & 0 \\ 0 & 0 \end{bmatrix} 
\end{pmatrix} \nonumber
\end{align}
}}
Adding them all up gives,
\begin{align}
\mathcal{L}_{i,k} + \mathcal{L}_{i,k}^T 
\leq
\begin{pmatrix}
\begin{bmatrix} \mathcal{L}^1_{i} & 0 \\ 0 & 0 \end{bmatrix} & 0 & 0 & 0 \\ 
0 & \mathcal{L}^2 & 0 & 0 \\ 
0 & 0 & 0 & 0 \\ 
0 & 0 & 0 & \begin{bmatrix} \mathcal{L}^3_i & 0 \\ 0 & 0 \end{bmatrix}
\end{pmatrix} \nonumber
\end{align}
with,
\begin{align}
 \mathcal{L}^1_{i} &= \lambda_1^{-1} P_{0,i} \mathcal{A} \mathcal{A}^T P_{0,i} + \lambda_3^{-1} P_{0,i} \mathcal{B} \mathcal{B}^T P_{0,i},  \nonumber\\
 \mathcal{L}^2 &= 
\begin{bmatrix}
(\lambda_1+\lambda_2)E_A^TE_A & 0 & 0 & 0 \\
0 & (\lambda_3+\lambda_4)E_B^TE_B & 0 & 0 \\
0 & 0 & 0 & 0 \\
0 & 0 & 0 & 0
\end{bmatrix}  \nonumber \\
 \mathcal{L}^3 &= \lambda_2^{-1} P_0 \mathcal{A} \mathcal{A}^T P_0 + \lambda_4^{-1} P_0 \mathcal{B} \mathcal{B}^T P_0  \label{eq_after_timevarying_bound}
\end{align}
This would lead to the inequality in \eqref{eq_lmi_split_constant_timevarying} to,
%{\scriptsize{
\begin{align}
\begin{pmatrix}
 \mathcal{L}^{11}_i & Q_{A,i} & \Phi_{i}^T P & 0  \\
   * & -\Gamma_2+\mathcal{L}^2 & 0 & Q_{B} \\
   * &   * & -P & 0 \\
   * &   * &  * & \mathcal{L}^{44}
\end{pmatrix} < 0
\end{align}
%}}
where,
\begin{align}
\mathcal{L}^{11}_i &= \begin{bmatrix} -P_0+I+\mathcal{L}_{i}^1  & 0 \\ 0 & P_1\end{bmatrix} \nonumber \\
\mathcal{L}^{44} &= \begin{bmatrix} -P_0+I+\mathcal{L}^3  & 0 \\ 0 & P_1\end{bmatrix}
\end{align}
These terms have quadratic entries that could be handled by applying Schur's complement. In this way, we could consider,
{\small{
\begin{align}
\mathcal{L}^{11}_i < 0 \Leftrightarrow  
\begin{bmatrix} 
-P_0+I  & 0 & P_{0,i}\mathcal{A} & P_{0,i}\mathcal{B} \\ 
0 & P_1 & 0 & 0 \\
0 & 0 & -\lambda_1 I & 0 \\
0 & 0 & 0 & -\lambda_3 I
\end{bmatrix} < 0
\end{align}
}}
Similarly for $\mathcal{L}^{44}$. By putting them together and rearranging, we get,

%{\small{
%\begin{align}
%\mathcal{L}^{44} < 0 \Leftrightarrow
%\begin{bmatrix}
%
%\end{bmatrix}
%\end{align}
%}}
\begin{align}
\begin{bmatrix}
-P+I & Q_{A,i}    & \Phi_i^T P & 0      &  \\
  *  & T_{22} &  0         & Q_B^T  &  T_{\mathcal{A}\mathcal{B}} \\
  *  &  *     &   -P       & 0      &  \\
  *  &  *     &    *       & -P     &  \\
     & \bigtranspose &     &       & \Lambda \\
\end{bmatrix} < 0 
\end{align}
with,
\begin{align}
T_{\mathcal{A}\mathcal{B}} &= \begin{bmatrix} 
P_{0,i} \mathcal{A} & P_{0,i} \mathcal{B} & 0 & 0 \\
0   &   0 & 0 & 0 \\
0   &   0 & 0 & 0 \\
0   &   0 & P_0 \mathcal{A} & P_0 \mathcal{B}
\end{bmatrix} \nonumber \\
\Lambda &= \begin{bmatrix}
-\lambda_1 I & 0 & 0 & 0 \\
0 & -\lambda_3I & 0 & 0 \\
0 & 0 & -\lambda_2I & 0 \\
0 & 0 & 0 & -\lambda_4I 
\end{bmatrix} \nonumber
\end{align}
%and are given in \eqref{eq_theorem_term_defs}. 
Hence the proof.
\end{proof}
\begin{corollary}{}
We could formulate the observer design as an optimization problem with the objective to minimize the $\mathbb{L}_2$-gain between the perturbation factors $\tilde{u}_k$ and the errors $\mathbf{e}_{a,k}$ in \eqref{eq_error_dyn_augmented}. We could aim to minimize a scalar $\beta$, such that,
\begin{align}
\underset{P_0, P_1, F_i, R_i, \Gamma_2^j, \lambda_m}{\min} \beta
\end{align}
$\forall i \in [1,r]$, $\forall j \in \{x,u,\theta,\Delta \theta\}$, $\forall m \in \{1,2,3,4\}$ such that, the LMIs in \eqref{eq_main_lmi_condition} are satisfied along with,
\begin{align}
\beta I > \Gamma_2^j,\ \ \forall j \in \{x,u,\theta,\Delta \theta\} \label{eq_beta_Gamma_inequality}
\end{align}
There is an inherent assumption that the $\mathbb{L}_2$-gains of various perturbations are scaled appropriately so that using a single $\beta$ makes sense. This could otherwise be achieved by using appropriate scaling factor instead of $I$ on the left hand side of\eqref{eq_beta_Gamma_inequality}.
\end{corollary}
\begin{corollary}{}
Measurement noise could be added to the output equation in \eqref{eq_sys_model_lin_final} so that, 
\begin{align}
\mathbf{y}_k = C \mathbf{x}_k + H \mathbf{\nu}_k
\end{align}
where, $\nu_k \in \mathbb{R}^{n_y}$ is the measurement noise with $H \in \mathbb{R}^{n_y\times n_y}$ the transmission matrix. This would lead to the perturbation variable to become $\tilde{u}_k = \begin{bmatrix} \mathbf{x}_k & \mathbf{u}_k & \Theta_k & \Delta \Theta_k & \mathbf{\nu}_k \end{bmatrix}^T$ and the matrix $\Psi_{i,k}$ in \eqref{eq_phi_psi} as,
\begin{align}
\Psi_{i,k} &= \begin{bmatrix}
\Delta A_k & \Delta B_k & 0 & 0 & -L_iH \\
0 & 0 & I+K_{\theta} & I & -K_{y,i}H
\end{bmatrix}
\end{align}
which would then lead to the same LMIs with the modifications in the following components in \eqref{eq_theorem_term_defs} as,
{\footnotesize{
\begin{align}
Q_{A,i} &= \begin{bmatrix}
0 & 0 & -C^TF_{i}^T (I+K_{\theta}) & -C^T F_{i}^T & -R_i H\\
& & & & & \\
0 & 0 & -K_{\theta}^T P_1 (I+K_{\theta}) & -K_{\theta}^T P_1 & -F_i H
\end{bmatrix} \nonumber \\
T_{22} &= 
\begin{bmatrix}
T_{22}^{11} & 0 & 0 & 0 & 0\\
0 & T_{22}^{22} & 0 & 0 & 0\\
0 & 0 & -\Gamma_2^\theta & 0 & 0\\
0 & 0 & 0 & -\Gamma_2^{\Delta \theta} & 0 \\
0 & 0 & 0 & 0 & -\Gamma_2^{\nu}
\end{bmatrix}
\end{align}
}}%
where, $\Gamma_2^\nu$ is the $\mathbb{L}_2$-gain between the noise $\nu$ and the error $\mathbf{e}_{a,k}$. It is to be noted that $\Gamma_2^\nu$ will also be added as a diagonal block in the matrix $\Gamma_2$.
\end{corollary}
\begin{remark}{}
It is to be noted that since there are only two time-varying terms $\Delta A_k$ and $\Delta B_k$ in \eqref{eq_time_varying_terms_matrix}, we could split the time-varying terms into only two additive factors and hence apply the Lemma \ref{lemma_timevaring_bounds} twice. However, the resulting matrix inequality is nonlinear with crossover terms making it impossible to resolve. Hence four additive factors were used.
\end{remark}
\begin{remark}{}
In the continuous-time version in \cite{bezzaoucha2013state}, the factor $K_{\theta}$ allowed to avoid numerical issues in the LMI conditions. In our work, this has been followed through. The value of $K_{\theta}$ however, is also important because it may lead to the effect of the innovation term $K_{y,i}(\mathbf{y}_k-\mathbf{\hat{y}}_k)$ to become negligible due to relative scaling between $K_{\theta}$ and $K_{y,i}$ as discussed in \cite{srinivasarengan2016nonlinear}. This could be done by adding an extra condition. For example, for a scalar parameter estimation case, let $k_\theta$ be the scalar value of the observer gain $K_\theta$, which would lead to the condition, 
\begin{align}
\frac{1}{k_{\theta}} K_{y,i}> \rho
\end{align}
where $\rho>1$ is a constant chosen depending upon the relative scaling between $\theta$ and $\mathbf{y}_k-\mathbf{\hat{y}}_k$. Along with the LMIs in \eqref{eq_main_lmi_condition}, we could add, for a scalar case,
\begin{align}
F_i > \rho P_1 k_{\theta}
\end{align}
\end{remark}
\begin{remark}{}\label{remark_Qe}
The LMI in \eqref{eq_main_lmi_condition} could be considered restrictive partly because of the term $-P+I$ that calls for $P$ to be more positive than $I$. This starts with the term, $e_{a,k}^T I e_{a,k}$ in the Lyapunov function trajectory in \eqref{eq_lyapunov_trajectory}. If a solution is unavailable for this case, we could replace this with $e_{a,k}^T Q_e e_{a,k}$, where $Q_e$ could be chosen to be a value that allows for a solution to the LMI \eqref{eq_main_lmi_condition} exists.
\end{remark}
\begin{remark}{}
The extension of the Theorem \ref{thm_linear_main} for nonlinear system represented by T-S models is straightforward. That is the nonlinear model has to be transformed as a T-S model with time-varying matrices and then all the proposed development with the introduction of a supplementary index for some of the matrices involved. Further the weighting functions now would be the products of membership functions of both unknown parameters as well as the premise variables of the T-S model. And hence the components of the uncertain-like terms would be different.
\end{remark}
\section{Adaptive Observer Design} \label{sec_adapt_observer_design}
As discussed in Sec. \ref{sec_introduction}, adaptive observers deal with cases where the unknown parameter is a constant, that is $\Theta_{k+1} = \Theta_k = \Theta$. For the discrete-time version of the nonlinear system in \cite{cho_systematic_1997} of the form,
\begin{align}
\mathbf{x}_{k+1} &= A \mathbf{x}_k + \phi(\mathbf{x}_k,\mathbf{u}_k) + b f(\mathbf{x}_k,\mathbf{u}_k) \Theta \nonumber \\
\mathbf{y}_k &= C \mathbf{x}_k \label{eq_nonlin_adapt_obs_form_discrete}
\end{align}
a quasi-LPV equivalent would be of the form \eqref{eq_ts_with_theta_initial}. Our aim is to design an adaptive observer for this model, assuming that we know a range of values $[\theta_i^1, \theta_i^2]$ in which the true value of each of the $\theta_i, \forall i \in 1,..,n_\theta$ lies. This substitutes for the sector bounds for the time-varying case. Following the same argument of applying SNL transformation for the time-varying parameter in the previous section with these bounds, we obtain,
\begin{align}
\mathbf{x}_{k+1} &= \sum_{i=1}^{s} h_i(\mathbf{z}_k,\Theta_k) (A_i \mathbf{x}_k + B_i \mathbf{u}_k) \nonumber \\
\mathbf{y}_k &= C \mathbf{x}_k \label{eq_sys_model_ts_final}
\end{align}
where $s=2^{n_p+n_\theta}$ and $h_i(\mathbf{z}_k,\Theta_k)$ is the weighting function obtained by normalizing the product of membership functions associated with the premise variables $\mathbf{z}_k$ and the parameters $\theta_k$. For this type of system, we propose an observer of the form,
\begin{align}
\mathbf{\hat{x}}_{k+1} &= \sum_{i=1}^s h_i(\hat{z}_k, \hat{\Theta}_k) \left[ A_i \hat{x} + B_i \mathbf{u}_k + L_i(\mathbf{y}_k-\mathbf{\hat{y}}_k) \right] \nonumber \\
\hat{\Theta}_{k+1} &= \hat{\Theta}_k + \sum_{i=1}^s h_i(\hat{z}_k, \hat{\Theta}_k) K_{y,i}(\mathbf{y}_k-\mathbf{\hat{y}}_k) \nonumber \\
\mathbf{\hat{y}}_k &= C \mathbf{x}_k \label{eq_adaptive_observer_model}
\end{align}
As could be noted, the $K_{\theta}$ gain term has been dropped. One main reason is the simplification this offers (would be apparent soon). Further, the condition $K_{\theta}=0$ lead to unsolvable LMIs in the continuous-time case, but not in discrete-time. To compute the state and parameter error, we follow the uncertain-like model approach, the augmented error dynamics is given by,
\begin{align}
\mathbf{e}_{a,k+1} &= \begin{bmatrix}
A_i-L_iC & 0 \\ -K_{y,i} C & 0 
\end{bmatrix}
\mathbf{e}_{a,k}
+
\begin{bmatrix}
\Delta A_k & \Delta B_k \\ 0 & 0
\end{bmatrix}
\begin{bmatrix}
\mathbf{x}_k \\ \mathbf{u}_k
\end{bmatrix} \nonumber
\end{align}
As could be seen, the number of perturbations has reduced and hence the dynamics matrices simplified. By applying discrete-time BRL and following it up with the application of LMI equivalence using Schur's complement (Lemma \ref{lemma_schur}), and then splitting the Lyapunov matrix to be of the form $P = \begin{bmatrix} P_0 & 0 \\ 0 & P_1 \end{bmatrix}$, and applying the variable transformations $R_i = P_0 L_i$ and $F_i = P_1 K_{y,i}$, we get,
\begin{align}
\begin{bmatrix}
-P+I & \Phi_{i}^T P \Psi_{i,k} & \Phi_{i}^T P & 0  \\
  *    & -\Gamma_2 & 0 & \Psi_{i,k}^T P \\
  *    & * & -P & 0 \\
  *    & * & * & -P
\end{bmatrix} < 0
\end{align}
with
\begin{align}
\Phi_i^T P \Psi_{i,k} &= \begin{bmatrix} P_{0,i} \Delta A_k & P_{0,i} \Delta B_k \\ 0 & 0  \end{bmatrix} \nonumber \\
\Phi_i^T P &= \begin{bmatrix} P_{0,i} & -C^TF_i^T \\ 0 & 0 \end{bmatrix} \nonumber \\
\Psi_{i,k}^T P &= \begin{bmatrix} \Delta A_k^T P_0 & \Delta B_k^T P_0 \\ 0 & 0\end{bmatrix} \label{eq_new_phi_p_etc_matrices}
\end{align}
Further following the same steps described in the proof of Theorem \ref{thm_linear_main}, we can summarize the results as follows,
\begin{theorem}{}
Given a nonlinear discrete-time system of the form \eqref{eq_nonlin_adapt_obs_form_discrete} which can be transformed to a T-S model, we could design an observer of the form \eqref{eq_adaptive_observer_model}, if there exists $P_0$, $P_1$, $R_{i}$, $F_{i}$, $\lambda_1$, $\lambda_2$, $\lambda_3$, $\lambda_4$, $\Gamma_2^j$ ($\forall i \in [1,r], \forall j\in \{x,u\}$), such that,
\begin{align}
&\qquad P_0 =P_0^T > 0, \ P_1 = P_1^T> 0 \\
&\qquad\lambda_m>0,\ \forall m\in{1,2,3,4}, \ \Gamma_2^j>0, \forall j \\
&\begin{pmatrix}
-P+I & 0      & \Phi_i^T P & 0    &  \\
  *  & T_{22} &  0         & 0    &  T_{\mathcal{A}\mathcal{B}} \\
  *  &  *     &   -P       & 0    &  \\
  *  &  *     &    *       & -P   &  \\
     & \bigtranspose &     &      & \Lambda \\
\end{pmatrix} < 0 \label{eq_main_lmi_condition_adaptive_observer}
\end{align}
where $\Phi_i^T P$ is given in \eqref{eq_new_phi_p_etc_matrices}, $\Lambda$ is given in \eqref{eq_theorem_term_defs}, $T_{\mathcal{A}\mathcal{B}}$ has the same structure as in  \eqref{eq_theorem_term_defs} except to accommodate the changes in the number of zero rows due to the change of $T_{22}$ to,
\begin{align}
T_{22} &=
\begin{bmatrix}
T_{22}^{11} & 0  \\
0 & T_{22}^{22}  \\
\end{bmatrix}\nonumber
\end{align}
where, $T_{22}^{11} = -\Gamma_2^x + (\lambda_1+\lambda_2) E_A^T E_A$ and $T_{22}^{11} =  -\Gamma_2^u + (\lambda_3+\lambda_4) E_B^T E_B$.
\end{theorem}
\begin{proof}{}
The proof follows that of Theorem \ref{thm_linear_main} with the changes in the matrix block entries discussed above.
\end{proof}
\section{Simulation Example} \label{sec_sim_example}
We consider the discrete-time version of the simplified waste water treatment plant from \cite{bezzaoucha2013state}. The simplification concerns reducing the 10-state system to a 2-state model, given by,
\begin{align}
x_{1,k+1} &= x_{1,k} + T_s \left[ \frac{a x_{1,k}}{x_{2,k}+b} x_{2,k} - x_{1,k} u_k \right] \nonumber \\
x_{2,k+1} &= x_{2,k} + T_s \left[ -\frac{c a x_{1,k}}{x_{2,k}+b} x_{2,k} + (d-x_{2,k}) u_k\right] \nonumber \\
y_k &= x_{1,k}
\end{align}
In this model, we consider an uncertainty in the parameter $a$, that is,
\begin{align}
a = a_0 + \theta
\end{align}
The parameters of the model used are given in Table \ref{tab_model_parameters}.
\begin{table}
\centering
\caption{Model Parameters}
\begin{tabular}{|c|c|}
\hline
\textbf{Parameter} & \textbf{Value} \\
\hline
$a_0$ & $0.5$\\
$b$ & $0.4$\\
$c$ & $0.4$\\
$d$ & $2$\\
$T_s$ & $1$\\
\hline
\end{tabular}
\label{tab_model_parameters}
\end{table}
\begin{figure}
\centering
\includegraphics[scale=0.35]{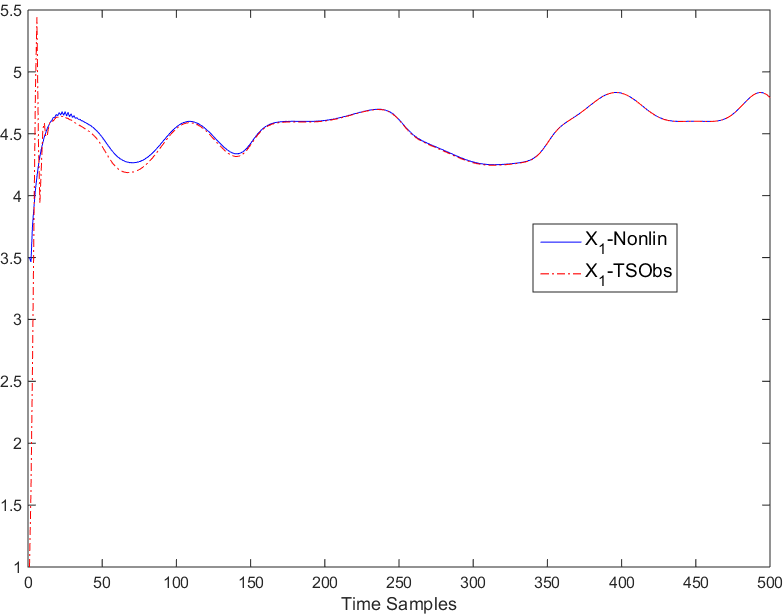}
\caption{Estimation of $x_{1,k}$}
\label{fig_state_1_estimation}
\end{figure}%
The unknown parameter $\theta$ in \eqref{eq_adaptive_observer_model} is constant; however, for the observer design purposes we assume it to be known and in the range,
\begin{align}
\theta \in [-0.3, 0.3]
\end{align}
We choose the premise variables, 
\begin{align}
z_{1,k} = u_k, \qquad \text{and}, \qquad z_{2,k} = \frac{x_{1,k}}{x_{2,k}+b}
\end{align} 
It is evident that $z_{2,k}$ depends on the unmeasured state $x_{2,k}$ making it unmeasured. Assuming a range of values for $u_k \in [0, 0.4]$, $x_{1,k} \in [0.01,6]$, and $x_{2,k} \in [0.01, 3]$, we get the range of the premise variables as, 
\begin{align}
z_{1,k} \in [0,0.4], \qquad z_{2,k} \in [0.003, 14.63]
\end{align}
\begin{figure}
\centering
\includegraphics[scale=0.35]{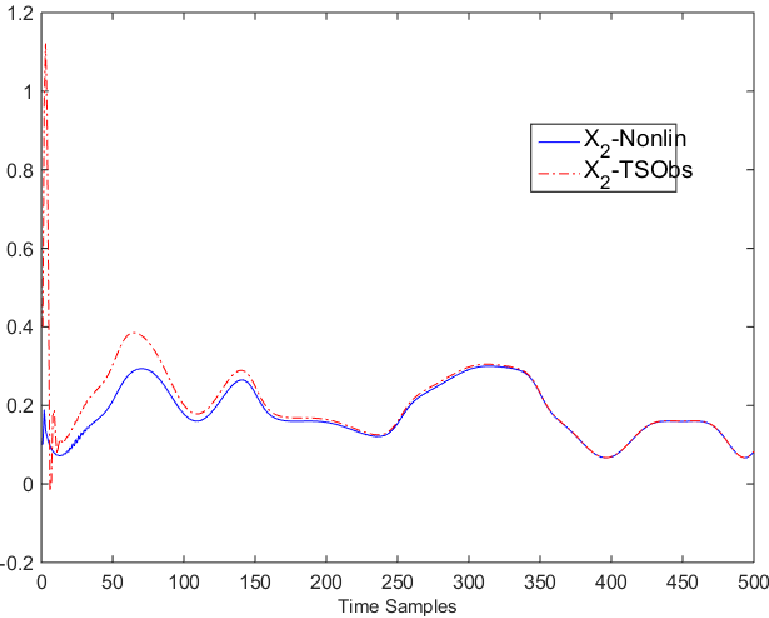}
\caption{Estimation of $x_{2,k}$}
\label{fig_state_2_estimation}
\end{figure}
With these parameters, we get the model
\begin{align}
\mathbf{x}_{k+1} &= \sum_{i=1}^8 h_i(\mathbf{z}_k,\theta_k) \left[A_i \mathbf{x}_k + B_i u_k \right]
\end{align}
where $h_i(\mathbf{z}_k,\theta)$ is obtained from the product of membership functions of $z_{1,k}$, $z_{2,k}$ and $\theta$ corresponding to the submodel $i$. The system matrices are given by,
{\scriptsize{
\begin{align}
A_1 &=
\begin{bmatrix}
1 & 5.9\times 10^{-4} \\ 0 & 0.99
\end{bmatrix}
A_2 =
\begin{bmatrix}
1 & 0.0024 \\ 0 & 0.99
\end{bmatrix}
A_3 = 
\begin{bmatrix}
1 & 2.92 \\ 0 & -0.17
\end{bmatrix} \nonumber \\
A_4 &= 
\begin{bmatrix}
1 & 11.7 \\ 0 & -3.68
\end{bmatrix}
A_5 = 
\begin{bmatrix}
0.6 & 5.9\times 10^{-4} \\ 0 & 0.6
\end{bmatrix}
A_6 = 
\begin{bmatrix}
0.6& 0.0024 \\ 0 & 0.6
\end{bmatrix} \nonumber \\
A_7 &=
\begin{bmatrix}
0.6 & 2.93 \\ 0 & -0.57
\end{bmatrix}
A_8 =
\begin{bmatrix}
0.6 & 11.7 \\ 0 & -4.1
\end{bmatrix} \ 
\text{and}, \qquad
B_i = \begin{bmatrix}
0 \\ 2
\end{bmatrix},\ \forall i \nonumber
\end{align}
}}%
\begin{figure}[h]
\centering
\includegraphics[scale=0.45]{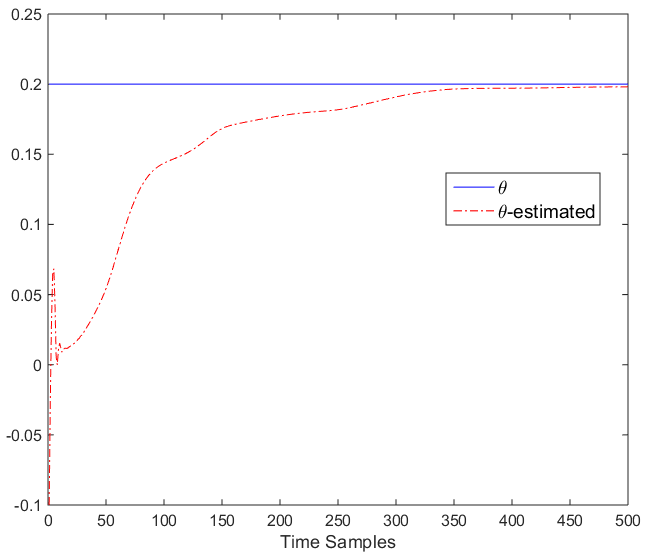}
\caption{Estimation of $\theta_{k}$}
\label{fig_para_estimation}
\end{figure}%
We used Matlab with Yalmip \cite{lofberg2004yalmip} interface with LMIlab toolbox to solve the LMI conditions. 
\begin{remark}{}
Some problem specific conditions could be added to obtain an optimum solution to the problem. For example, pole placement for the state observers $A_i-L_iC$ could be added as a separate LMI condition so as to achieve favourable rate of convergence. Further, some minimum value for the gain corresponding to parameter estimation, $K_i, \forall i$ could be imposed so that the innovation term is useful in augmenting the estimated $\theta$ (due to the relative scaling between the values of $\hat{\theta}$ and $y-\hat{y}$).  Further, as noted in Remark \ref{remark_Qe}, the value of $Q_e$ was chosen as 
\begin{align*}
Q_e = \begin{bmatrix} 0.001I_{n_x} & 0 \\ 0 & 0.1I_{n_{\theta}}
\end{bmatrix}
\end{align*} 
\end{remark}
\begin{remark}{}
It is to be noted that there are a number of variables to be determined by the LMI solver. This could be reduced by fixing some of the parameters. For this example, we chose the values for $\Gamma_2^x=\Gamma_2^u=0.1$ and $\lambda_i=0.001,\ \forall i=1,2,3,4$. The significantly reduces the computational complexity of the problem.
\end{remark}
With the above conditions, we obtain the following observer gain values, \par
{\scriptsize{
\begin{align}
L_1 &= \begin{bmatrix} 0.23 \\ 0.24 \end{bmatrix}
L_2 = \begin{bmatrix} 0.23 \\ 0.24 \end{bmatrix}
L_3 = \begin{bmatrix} 0.31 \\ 0.21 \end{bmatrix}
L_4 = \begin{bmatrix} 0.61 \\ 0.09 \end{bmatrix} \nonumber \\
L_5 &= \begin{bmatrix} -0.41 \\ 0.30\end{bmatrix}
L_6 = \begin{bmatrix} -0.41 \\ 0.30\end{bmatrix}
L_7 = \begin{bmatrix} -0.41 \\ 0.30 \end{bmatrix}
L_8 = \begin{bmatrix} -0.27 \\ 0.24 \end{bmatrix} \nonumber
\end{align}
}}
and $K_i = 0.03,\ \forall i$.  

With these gain values, we obtain the state estimation results as shown in the Fig \ref{fig_state_1_estimation} and \ref{fig_state_2_estimation}. The estimation of the unknown parameter is given in Fig \ref{fig_para_estimation}.
Further, the input used for the simulation is shown in the Fig \ref{fig_input}.
\begin{figure}
\centering
\includegraphics[scale=0.45]{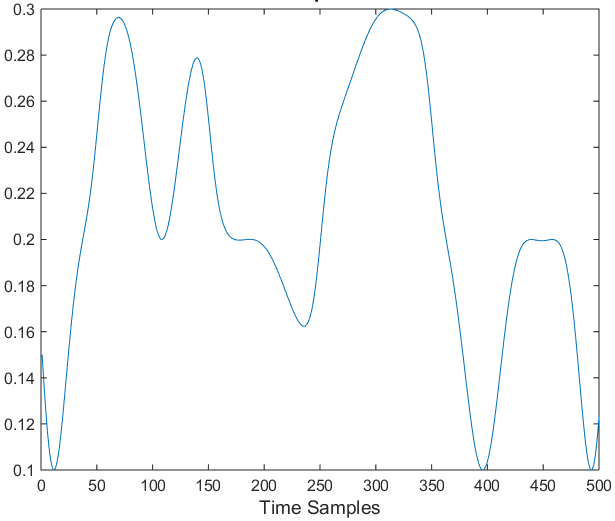}
\caption{Input used for simulation}
\label{fig_input}
\end{figure}
To illustrate the nonlinearity of the model, we show the variation of the weighting function $h_i(\hat{x},\hat{\theta})$ in the Fig \ref{fig_weighting_functions}.
\section{Concluding Remarks} \label{sec_concluding_remarks}
We presented an adaptive observer design procedure for discrete-time nonlinear systems which could be converted to a quasi-LPV form. The presented work fills a gap for adaptive observer design for discrete-time nonlinear systems for those cases where the transformed quasi-LPV matrices depend on one of the unmeasured states of the system. The results in this approach are conservative, but provides an opening in exploring a systematic approach for adaptive observer design for nonlinear system of this type. One interesting extension could be to explore other observer structures, especially the one proposed in \cite{guyader2003adaptive} and follow the similar design strategy of our work so as to expand for unmeasured premise variable case.
\begin{figure}
\centering
\includegraphics[scale=0.4]{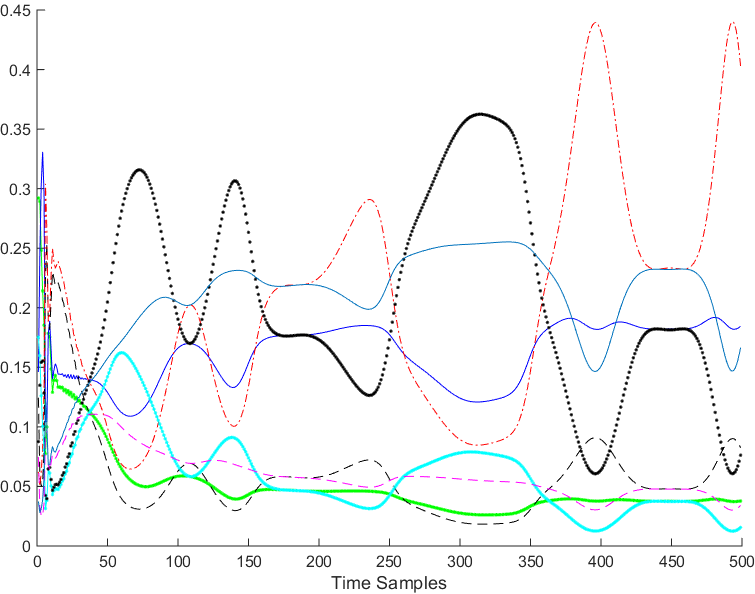}
\caption{Weighting functions of the submodels}
\label{fig_weighting_functions}
\end{figure}

\bibliography{./../../../../researchRef,./../../../../nonlinAdaptObs}

% Bibliography is different
% 
\end{document}